\documentclass{amsart}
\usepackage{lmodern}			% Usa a fonte Latin Modern
\usepackage[T1]{fontenc}		% Selecao de codigos de fonte.
\usepackage[utf8]{inputenc}		% Codificacao do documento (conversão automática dos acentos)
\usepackage{indentfirst}		% Indenta o primeiro parágrafo de cada seção.
\usepackage{nomencl} 			% Lista de simbolos
\usepackage{color}				% Controle das cores
\usepackage{graphicx}			% Inclusão de gráficos
\usepackage{microtype} 			% para melhorias de justificação
\usepackage{amsfonts}
\usepackage{amsmath}
\usepackage{amsthm}
\newtheorem{theorem}{Theorem}
\newtheorem{definition}{Definition}
\newtheorem{obs}{Comment}
\newtheorem{ex}{Example}
\newtheorem{corol}{Corollary}
\title{Algebraic Vector Space}
\author{Fernando M. Matias}
\address{IFRJ, Campus Niter\'oi, RJ, Brazil}
\email{fernando.matias@ifrj.edu.br.}
\date{\today, v-1.1}
\definecolor{blue}{RGB}{41,5,195}

\begin{document}

% ----------------------------------------------------------
% ELEMENTOS PRÉ-TEXTUAIS
% ----------------------------------------------------------
%---
\begin{abstract}
	We show that the definition of an algebraic basis for a vector space allows the construction of an isomorphism with the one here called Algebraic Vector Space. Although the concept does not bring anything new, we mention some of the problems that the language established here can inspire.
%	\noindent
\end{abstract}
\maketitle

% titulo em outro idioma (opcional)

% resumo em inglês
%\renewcommand{\resumoname}{Abstract}

% ]  				% FIM DE ARTIGO EM DUAS COLUNAS
% ---

\begin{center}\smaller
%{Data de submissão e aprovação}: elemento obrigatório. Indicar dia, mês e ano

%\textbf{Identificação e disponibilidade}: elemento opcional. Pode ser indicado o endereço eletrônico, DOI, suportes e outras informações relativas ao acesso.
\end{center}
% ----------------------------------------------------------
% Introdução
% ----------------------------------------------------------
\section{Introduction}
This work defines what we call \textbf {Algebraic Vector Space}. The main purpose of this construction is to serve as the basis for another article \cite{MatAlg}.

 We built the concept of Algebraic Vector Space in Sec.  \ref{sec:Const} and \ref{sec:form}. First in a conceptual way, then with the necessary rigor when dealing with uncountable sets. We show that the use of algebraic bases allows manipulation with a “finite appearance” even in the case of vector spaces with an uncountable basis.

In Sec. \ref{sec:Ex} we show that there is nothing new in the approach used. We show through some examples that, in fact, we only give “old acquaintances” a new look.

In Sec. \ref{sec:Pot} we explain why, even if we do not add something that deserves the name "new", we believe that the making of this work is valid. We do this by showing some of the possible lines of developments that follow naturally from this article.
\section{Inception} \label{sec:Const}
Let $ B = \{\mathbf {u}_\lambda \}_{\lambda \in L} $ be an algebraic basis for the vector space $ E $ over some field $ \mathbb {K} $.
\footnote{The algebraic basis is also known as the Hamel basis. Its existence is guaranteed by the Zorn's lemma \cite{kreyszig} and in it, every vector is uniquely expressed as a finite linear combination of elements of this basis.}
\footnote{For the sake of convenience, associated with the notation developed in this article, there will be a single restriction on the index set: we can't have $\phi \in L $. There is no problem. If any $ L $ has the empty set as an element, just follow a simple procedure. As in ZFC set theory, there is no universal set, there will be some set $ c $ such that $ c \notin L $. We then form the set $ L '= (L /\{\phi \}) \cup \{c\} $ and use $ L' $ as the new index set.}
If $ E $ is finite-dimensional, the index set $ L $ will be finite, but in the most general case, it may even be uncountable.
\footnote{It is implicit in the index set concept that it turns $ B $ into a well-ordered set. The guarantee that it is always possible to index the elements of any non-empty set is given by the Well-ordering theorem \cite{halmos_naive}. In the case of the trivial vector space, $ E = \mathbf {0} $, where we have $ B = \phi $, just make $ L = \phi $. The vacuous truth will ensure that the statements made in this work about sets of indexes become true even in the trivial case.}
By construction, any vector can be written as a finite linear combination of the elements of an algebraic basis. If we consider only linear combinations with nonzero coefficients, we guarantee the uniqueness of the decomposition of nonzero vectors. In other words, whatever $ \mathbf{x}\in E $, $ \mathbf{x} \neq \mathbf{0} $, there will be a single ordered n-tuple of ordered pairs $ ((\alpha_1, \lambda_1), (\alpha_2, \lambda_2), ..., (\alpha_n, \lambda_n)) $, $ \alpha_i \neq 0 $ and $ \lambda_1 <\lambda_2<... <\lambda_n $, such that
\begin{align}\label{eq:base_algebrica}
 \mathbf{x}= \alpha_1\cdot \mathbf{u}_{\lambda_1}+ \alpha_2\cdot \mathbf{u}_{\lambda_2}+...+\alpha_n\cdot \mathbf{u}_{\lambda_n}
\end{align}
We will rename some indexes on equation \ref{eq:base_algebrica}
\begin{align}
\mathbf{x}= \alpha_{\lambda_1}\cdot \mathbf{u}_{\lambda_1}+ \alpha_{\lambda_2}\cdot \mathbf{u}_{\lambda_2}+...+\alpha_{\lambda_n}\cdot \mathbf{u}_{\lambda_n}
\end{align}
and get a single ordered pair
$\mathbf{x}_B=(\{\lambda_1,\lambda_2,...,\lambda_n\},(\alpha_{\lambda_1},\alpha_{\lambda_2},...,\alpha_{\lambda_n})) $ for each $\mathbf{x}\neq \mathbf{0}$. Note that, by construction, if  $ \mathbf{x}\neq\mathbf{0} $ we will have $\mathbf{x}_B\neq (\phi,\phi)=\{\phi\} $. So it is natural to associate the null vector with $\mathbf{0}_B= (\phi,\phi) $.

This informal construction defines a bijection between any vector space and a special set of ordered pairs, which we call $ E_B $. This bijection induces natural operations on  $ E_B $ that transforms the bijection into an isomorphism. In the next section, we will build these operations, but only after completing the formal construction of the $ E_B $ set. The objective is to show that, however general the vector space may be, it will always be possible to construct its algebraic isomorphic. Care is necessary because we even cover cases where the basis for the vector space can be an uncountable set.
\section{Formal Construction of Algebraic Vector Spaces}\label{sec:form}
We will use the notation $ \mathcal{P}_n(L)$ to refer to the set of subsets of $ L$ with $ n $ elements.
\begin{definition}
	Let  $ E $,  vector space over a field  $ \mathbb{K} $, and $ B $, a basis for $ E $ indexed by the $ L $ set. The  algebraic vectors of $ E $ on  $ B $, and denoted  $ E_B $, is defined as
	\begin{align*}
	&E^*_B=\left\{
	\begin{aligned}
	& \bigcup_{n=1}^{N} \mathcal{P}_n(L)\times(\mathbb{K}^*)^n ,\textrm{ if the cardinality of } L \textrm{ is }N\in\mathbb{N}\\
	&\bigcup_{n\in\mathbb{N}^*}\mathcal{P}_n(L)\times(\mathbb{K}^*)^n ,\textrm{ otherwise.}
	\end{aligned}
	\right.\\
	&E_B=E^*_B\cup (\phi,\phi)
	\end{align*}
\end{definition}
If $ B=\phi $, then $ L=\phi $ and $ E= (\phi,\phi)$. This is an expected result because the only vector space with $ B=\phi $ is the trivial one.
\begin{definition}
	Let $ \mathbf{v}\in E_B $. The  index set of $ \mathbf{v}  $ is defined as $ L_  \mathbf{v}=\pi_1(\mathbf{v}) $.
\end{definition}
\begin{definition}
	Let $ \mathbf{v}\in E_B $, $ \mathbf{v}\neq\mathbf{0} $,  and $ \lambda_i\in  L_  \mathbf{v} $.  The $ i^{th}$ component of  $ \mathbf{v}  $ is defined as  $ v^{\lambda_i}=\pi_i\circ\pi_2(\mathbf{v}) $. However, when there is no confusion, we will refer to $v^{\lambda_i}$ as $ v^i $.
\end{definition}
It is important to emphasize the restriction of this definition to non-null vectors. This makes $v^i\neq 0 $ for all $ \mathbf{v}\in E_B $ and $ i\in L$. This is natural in this construction because the null vector has no components. If we wanted to force the notation syntax, including the null vector in the component concept, we would only get $ 0^\phi=\phi $.
\begin{ex}
	If $ \mathbf{v}=((\alpha_{\lambda_1},\alpha_{\lambda_2},...,\alpha_{\lambda_n}),\{\lambda_1,\lambda_2,...,\lambda_n\}) $, then $L_  \mathbf{v}= \{\lambda_1,\lambda_2,...,\lambda_n\}$ and $ v^i=\alpha_{\lambda_i} $. Peculiarly $L_  \mathbf{0}= \phi$. So, we can write $ \mathbf{v}=((\alpha_{\lambda_1},\alpha_{\lambda_2},...,\alpha_{\lambda_n}),L_\mathbf{v}) $.
\end{ex}
\begin{definition}\label{def:MultPorEscalar}
	Let $ \kappa\in \mathbb{K} $ and $ \mathbf{v}=(L_\mathbf{v},(\alpha_{\lambda_1},\alpha_{\lambda_2},...,\alpha_{\lambda_n}))\in E_B $. So, the operation $ \boldsymbol{\cdot}:\mathbb{K}\times E_B\to E_B $ is defined as
	\begin{align}
\kappa\cdot\mathbf{v}=\left\{
\begin{aligned}
&(L_\mathbf{v},(\kappa\alpha_{\lambda_1},\kappa\alpha_{\lambda_2},...,\kappa\alpha_{\lambda_n})),\textrm{ if } \kappa\neq 0 \textrm{ and } \mathbf{v}\neq \mathbf{0}\\
&\mathbf{0},\textrm{ if } \kappa=0 \textrm{ or } \mathbf{v}= \mathbf{0}
\end{aligned}
\right.
\end{align}	
\end{definition}
\begin{corol}
	$ 1\cdot\mathbf{v}=\mathbf{v} $
\end{corol}
\begin{obs}
	As usual, we will write $\kappa\mathbf{v}=\kappa\cdot\mathbf{v}  $ and $ -1\cdot\mathbf{v}=-\mathbf{v} $.
\end{obs}
\begin{corol}\label{cor:{Compat}}
	Let $ \alpha,\beta \in \mathbb{K} $, and $ \mathbf{v}\in E_B $. Because the product on $ \mathbb{K} $ is associative, then $ \alpha(\beta\mathbf{v})=(\alpha\beta)\mathbf{v}$.
\end{corol}
\begin{definition}
	Let $ \mathbf{v},\mathbf{w}\in E_B $. We will denote $ 0_{\mathbf{v}\mathbf{w}}=\{i \in L_{\mathbf{v}}\cap L_{\mathbf{w}}|v^i+w^i=0\} $.
\end{definition}
\begin{definition} \label{def:SomaVetorial}
	Let $\mathbf{v}=(L_\mathbf{v},(\alpha_{\nu_1},\alpha_{\nu_2},...,\alpha_{\nu_n})) $  and $\mathbf{w}=(L_\mathbf{w},(\beta_{\mu_1},\beta_{\mu_2},...,\beta_{\mu_m})) $. So, the operation  $ + : 	E_B \times 	E_B\to	E_B$ is defined as $\mathbf{v}+\mathbf{w}=(L_{\mathbf{v}+\mathbf{w}},(\gamma_{\rho_1},\gamma_{\rho_2},...,\gamma_{\rho_p}))$, and
	\begin{align}
		 &L_{\mathbf{v}+\mathbf{w}}=(L_{\mathbf{v}}\cup L_{\mathbf{w}})/0_{\mathbf{v}\mathbf{w}}\\
		&\gamma_{\rho_i}=(v+w)^i=\left\{
		\begin{aligned}
		&\alpha_i,\textrm{ if } i\in L_{\mathbf{v}}/L_{\mathbf{w}}\\
		&\alpha_i+\beta_i,\textrm{ if } i\in (L_{\mathbf{v}}\cap L_{\mathbf{w}})/0_{\mathbf{v}\mathbf{w}}\\
		&\beta_i ,\textrm{ if } i\in L_{\mathbf{w}}/L_{\mathbf{v}}
		\end{aligned}
		\right.
	\end{align}
\end{definition}
Remarkably, the set $ 0_{\mathbf{v}\mathbf{w}} $ allows the immediate exclusion of all null components in the vector sum process. This is important because, by construction, the elements of $ E_B $ have no null components. In addition, if $ L_{\mathbf{v}+\mathbf{w}}=\phi  $ we cannot form any $ \gamma_{\rho_i} $, then $\mathbf{v}+\mathbf{w}=(\phi,\phi)=\mathbf{0} $.
By abuse of notation, and regarding definition \ref{def:SomaVetorial}, we will write
\begin{align}
(v+w)^i=v^i+w^i
\end{align}
\begin{theorem}
	The operation in definition \ref{def:SomaVetorial} turns $ (E_B,+) $ on a commutative group.
\end{theorem}
\begin{proof}
	Follows from respective properties on the field that the operation is commutative and associative. It also follows that $\mathbf{0}=(\phi,\phi) $ is a sum's neutral element. Moreover, $\mathbf{v}+(-\mathbf{v})=\mathbf{0}$, so any $ E_B $ element has his symmetric one.
\end{proof}
\begin{theorem}
	The operations $ \cdot $ and $ + $, on definitions \ref{def:MultPorEscalar} and \ref{def:SomaVetorial} turns $ (E_B,\cdot,+) $, on a vector space.
\end{theorem}
\begin{proof}
	By the previous theorem, we know that $ E_B $ with the  $ + $ operation is commutative. The \ref{cor:{Compat}} corollary exposes the scalar compatibility. Because the field's distributivity with the \ref{def:MultPorEscalar} and \ref{def:SomaVetorial} definitions, for all  $ \alpha,\beta\in\mathbb{K} $ and $ \mathbf{v},\mathbf{w}\in E_B $ we have $ (\alpha +\beta)\mathbf{u}=\alpha\mathbf{u}+\beta\mathbf{u} $ and $ \alpha(\mathbf{u}+\mathbf{w})= \alpha\mathbf{u}+\alpha\mathbf{w}$. So  $ (E_B,+) $ satisfies the vector space axioms.
\end{proof}
\begin{theorem}
$ E $ and $ E_B $ are isomorphic vector spaces.
\end{theorem}
\begin{proof}
	The choice of $ B $ as one algebraic basis for $ E $ allows writing any $ \mathbf{v}\in E $, $  \mathbf{v}\neq\mathbf{0} $ in a unique finite linear combination of elements from its basis. Ordering addiction by the index set $ L $, each sum will be associated with a unique element $ \mathbf{v}_B\in E_B $, $  \mathbf{v}_B\neq(\phi,\phi) $. Similarly, for each $ \mathbf{v}_B\in E_B $, $  \mathbf{v}_B\neq(\phi,\phi) $ we will build a unique finite linear combination with the basis elements. Because of basis definition, we will catch a unique element $ \mathbf{v}\in E $. So we have a one-to-one map between not nulls elements in the vector spaces $ E $ and $ E_B. $. If $ \mathbf{v}=\mathbf{0} $  we take $ \mathbf{0}_B=(\phi,\phi) $, completing the one-to-one map between $ E $ and $ E_B$.
	
	We show the bijection. Also, still using the decomposition in the ordered basis as a round-trip passage between $ E $ and $ E_B $, we see that the operations of the vector sum and scalar multiplication are preserved.
	
	So $ E $ and $ E_B $ are isomorphic vector spaces.
\end{proof}
\section{The Algebraic Vector Spaces Standard Basis}
The choice of a basis for a vector space defines its isomorphic algebraic. So, also induces over it a natural basis. Let $\mathbf{x}=(\{\lambda_1,\lambda_2,...,\lambda_n\},(\alpha_{\lambda_1},\alpha_{\lambda_2},...,\alpha_{\lambda_n})) $. Taking the definitions \ref{def:MultPorEscalar} and \ref{def:SomaVetorial} into account we see that
\begin{align*}
\mathbf{x}=\alpha_{\lambda_1}(\{\lambda_1\},1)+\alpha_{\lambda_2}(\{\lambda_2\},1)+...\alpha_{\lambda_n}(\{\lambda_n\},1)
\end{align*}
It's easy to see than the set $ \{(\{\lambda\},1)\}_{\lambda \in L} $  is the standard basis for algebraic vector spaces. We will use the usual notation, and write $ \mathbf{e}_\lambda=(\{\lambda\},1)$. So, as usual we can write
\begin{align*}
\mathbf{x}=\alpha_{\lambda_1} \mathbf{e}_{\lambda_1}+\alpha_{\lambda_2}\mathbf{e}_{\lambda_2}+...\alpha_{\lambda_n}\mathbf{e}_{\lambda_n}
\end{align*}
Putting this in a better form we obtain
\begin{align}\label{eq:DefSomaEinstein}
\mathbf{x}=\sum_{i\in L_\mathbf{x}}\alpha_i\mathbf{e}_i
\end{align}
In the case of $  L_\mathbf{x}=\phi $ we will make $ \sum_{i\in L_\mathbf{x}}\alpha_i\mathbf{e}_i=\mathbf{0} $. With all this we can make the bellow definition.
\begin{definition}[Extended Einstein's Notation]
An index variable that appears repeated in upper and lower positions, and is not otherwise defined, should be summed assuming each one of its valid values. The term "valid" meaning: that they're present at their intersection of the indexes sets involved. The result of that sum will be the suitable zero if this intersection is the empty set.
\end{definition}
Using this definition we can write
\begin{align*}
\mathbf{x}=\alpha^i\mathbf{e}_i
\end{align*}
The sum above is over the valid values for the index $ i $ in the sets $ L_\mathbf{x} $ and $ L $, whose intersection is just $ L_\mathbf{x} $. In the case of  $ L_\mathbf{x}=\mathbf{0} $, then $ L_\mathbf{x}=\phi $and we will obtain $\alpha^i\mathbf{e}_i=\mathbf{0}$. At this moment the only convenient zero is just the $ \mathbf{0} $ but there will be other possibilities in future works.
Finally, we will write this in a even more convenient form:
\begin{align}\label{eq:SomaEinstein}
\mathbf{x}=x^i\mathbf{e}_i
\end{align}
Usually the Einstein's notation is defined only in finite-dimensional metric spaces. At this moment we are dealing even with uncountable infinite-dimensional vector spaces without metric. There are many problems with careless uses of this notation, mainly associated to results base dependents. So, this also could occurs with the definition above. However, until now, this is not the case because the reasons below.
\begin{enumerate}
	\item We made only one restriction to the basis: it is algebraic. In the finite-dimensional cases all basis are algebraic, so, no really restriction are in fact made. Only in the infinite-dimensional case we have a real restriction, necessary to ensure well-defined sums. The consequences of this restriction will be discussed in a future work. For a while, we will assume this as a limitation.
	\item The main problems with this notation comes up when we use it to made “contractions”. In this cases is necessary to associate the position of the indexes with dual spaces, normally using a metric. For now, it's not a problem. We used this convention only to put upper indexes to components and lower to the basis elements. There is no contraction, so, there is no problem.
\end{enumerate}
The issues above will be re-analyzed in future work. For now, we will use unambiguous equation  \ref{eq:DefSomaEinstein} to define the notation that appears in \ref{eq:SomaEinstein}. In these terms, we are only using a still more compact notation.

\section{Demystifying Algebraic Vectors}\label{sec:Ex}

In the previous section, we showed that the definition of an algebraic basis over any vector space allows the formal construction of an isomorphism between vector spaces with the one here called “Space of Algebraic Vectors”. This construction, although excessively detailed due to the need for rigor, deals in fact with ridiculously simple concepts. This can be seen in some examples.
\begin{ex}[$ \mathbb{R}^3 $ Basis]
  $ \mathbb{R}^3 $ $ \{\mathbf{e}_1,\mathbf{e}_2,\mathbf{e}_3\} $ the canonical ordered basis of $ \mathbb{R}^3 $. 
	\begin{itemize}
		\item The usual basis for $ \mathbb{R}^3 $ is the set $ \{\mathbf{e}_1,\mathbf{e}_2,\mathbf{e}_3\} $.
		\item In the  algebraic vectors notation we have $ \{(\{1\},1),(\{2\},1),(\{3\},1)\} $.
		\item In the usual ordered tuple notation we have  $\{(1,0,0),(0,1,0),(0,0,1)\}$.
	\end{itemize}
\end{ex}
Until now, to choose is only matters of taste. So, let's examine other situations.
\begin{ex}[$ \mathbb{R}^3 $ Vectors]
Let $ \{\mathbf{e}_1,\mathbf{e}_2,\mathbf{e}_3\} $ the canonical ordered basis of $ \mathbb{R}^3 $. 
\begin{itemize}
	\item Let $ \mathbf{v}=2\mathbf{e}_1+\mathbf{e}_2+5\mathbf{e}_3$ and $  \mathbf{w}=4\mathbf{e}_1+\mathbf{e}_3$. So $ \mathbf{v}-5\mathbf{w}=-18\mathbf{e}_1+\mathbf{e}_2$.
	\item In the  algebraic vectors notation we have  $ \mathbf{v}=(\{1,2,3\},(2,1,5))$ e $  \mathbf{w}=(\{1,3\},(4,1))$. So $ \mathbf{v}-5\mathbf{w}=(\{1,2\},(-18,1))$.
	\item In the usual ordered tuple notation we have  $ \mathbf{v}=(2,1,5)$, $  \mathbf{w}=(4,0,1)$. So $ \mathbf{v}-5\mathbf{w}=(-18,1,0)$
\end{itemize}
\end{ex}
In this example, we note that algebraic vector notation is just a different, and less intuitive, way of writing something identical to those done in the usual approach. Furthermore, that the usual way via ordered tuples is much simpler. We will see that the following example changes the scenario a little.
\begin{ex}[$ \mathbb{R}^{100} $ Vectors]\label{ex:R100}
	Let $ \{\mathbf{e}_1,\mathbf{e}_2,...,\mathbf{e}_{100}\} $ the canonical ordered basis of $ \mathbb{R}^{100} $. 
	\begin{itemize}
		\item Let $ \mathbf{v}=2\mathbf{e}_{15}+\mathbf{e}_{54}+5\mathbf{e}_{83}$ e $  \mathbf{w}=4\mathbf{e}_{15}+\mathbf{e}_{83}$. So $ \mathbf{v}-5\mathbf{w}=-18\mathbf{e}_{15}+\mathbf{e}_{54}$.
		\item  In the  algebraic vectors notation we have  $ \mathbf{v}_B=(\{15,54,83\},(2,1,5))$ e $  \mathbf{w}_B=(\{15,54\},(4,1))$. So $ \mathbf{v}_B-5\mathbf{w}_B=(\{15,83\},(-18,1))$.
		\item  In the usual ordered tuple notation we would have $ \mathbf{v}$ with $ 97 $ zeros, in addition to the others $ 98 $ zeros at $  \mathbf{w}$. It is not necessary to write this to see that in this case the previous forms are better.
	\end{itemize}
\end{ex}
If for vector spaces with large numbers of elements in the basis it has already become obvious that the notation by ordered tuples is not efficient, there is still no great advantage in introducing these "new" vectors. The next example presents a new scenario.
\begin{ex}
Now we will see a vector space over the field  $ \mathbb{R}$  having $\mathbb{C}$ as the index set for his basis.
\footnote{The use of $\mathbb{C}$ as an index set may suggest a problem with the well-ordering of the basis since it is common to hear that complex numbers cannot be ordered. This is a common misconception. The correct thing to say is that the complexes do not admit an “ordered ring” structure, that is, there is no way to define a well-ordering that preserves the ring structure. However, as guaranteed by the Well-Ordering Theorem, the set itself admits ordering, even if it is one that does not preserve the ring. A classic example is the so-called lexicographic order: order first by the real part; if the real parts are the same, sort by the imaginary part. This ordering will not be problematic because we use the indexes only as labels, not using any operation that can usually be defined in this set. In practice here, it doesn't matter whether you work with $\mathbb{C}$ or $ \mathbb{R}^2$ as an index set.}
We can operate, without trouble, with vectors such as $ \mathbf{v}_B=(\{\sqrt{15},\sqrt{15}+2i,4\},(2,1,5))$ and $  \mathbf{w}_B=(\{\sqrt{15},4\},(4,1))$ getting $ \mathbf{v}_B-5\mathbf{w}_B=(\{\sqrt{15},4\},(-18,1))$.
\end{ex}
It would still be equally easy to solve the previous example in the usual way, by explicitly writing the respective linear combinations for the vectors. Perhaps there was some difficulty in writing and reading the indexes, but that is still not enough to justify the use of this notation. Much less for this detailed construction of a structure that, in practice, is still the same. Worse still, giving the pompous name “Space of Algebraic Vectors” to a mere change of notation. In the next section, we will explain the real motivation for this work.
\section{Algebraic Vectors Heuristic Potential}\label{sec:Pot}
 A search for “notation” in mathematical history books, such as \cite{boyer}, will be enough to understand his importance. Conceptual developments generate feedback in notation and vice-versa. This work does not dare to be proposing any “revolutionary” notation changes, but the retrospect shows that the exploration of new notations is something with the potential to be relevant.
 
 We have just created this new notation only as foundations for another work. We will use such a structure to extends the matrix definition for vector spaces with an uncountable basis in \cite{MatAlg}. However, this “preliminary results” generate opportunities by itself. In the following subsections, we will explore some of these possible developments.
 \subsection{Data structure and Sparse Matrices}
The search for a path to define matrices with an uncountable basis was the central objective of this work. Despite this, the most remarkable consequence surprisingly is obtained in the finite case. Such as we will see in \cite{SparseMatrices}, this notation generates a new way to the sparse matrices problem. As can be seen in \cite{esparsas1,esparsas2,esparsas3,esparsas4}, this is a subject with many applications and an active research theme.
 
In Computer Science terminology, we are using registers and pointers to describe vectors. The unmistakable case is in the vector $\mathbf{0}= (\phi,\phi)=\{\phi\} $. Its similarity with the null pointer is not a coincidence being the base of many approaches developed at this work. Transforming common vectors in the algebraic ones is a turning back from relational to the hierarchical data structures.

The algebraic vector notation may be useful in the analysis of the algorithmic efficiency as a function of the number of "zeros" present. I put the word in quotation marks because what is called "zero" can be seen in a more general context, for example, that of text fields filled with blanks.

The knowledge of relational algebra was fundamental for the development of relational databases. We don't presume to have done something so important like that, mainly because our approach offers no news on the case of finite-dimensional vector spaces. Even so, this analogy may be useful in the context of the Hierarchical and Multidimensional Databases, Data Science, Big Data, etc.

Maybe already exists a structure equivalent to that made in this article. However, and perhaps there is something new here, this work surrounded the finite case in a much broader conceptual context. Fortunately, something like this usually inspires new approaches
 \subsection{Graph theory}

Let's stay a little longer in the finite-dimensional case, where our approach shows a remarkable resemblance to Graph Theory. Initial explorations revealed the existence of a bijection from finite-dimensional vector spaces to low extension Directed Acyclic Graphs (DAG). We are looking for convenient operations that could be included in this set of DAGs to obtain a group isomorphism. We will name these new structures as  Algebraic DAGs Space (ADAGS).

The existence of optimal algorithms for DAG can inspire new approaches for sparse matrices issues with lots of repeated values. We have hope in ADAGS as a font for new methods in data mining issues. A manuscript about this is in preparation \cite{ArvAlg}.

\subsection{About the Dimensions of Vector Spaces with Uncountable Basis}
The cardinality of the indexes sets defines the dimension of vector spaces thanks to the existence of bijections between the different basis for each vector space. There are interesting questions about enumerable basis cases that become fascinating in the uncountable ones.

We build vector spaces over fields, and its choice determines the dimension of them. For example,  $ \mathbb{C} $ is one-dimensional over itself, two-dimensional over $  \mathbb{R} $ and infinite over $  \mathbb{Q} $. Do occur something similar in the denumerable basis case?  And in the case of basis with even bigger cardinalities?

Interesting questions arise just taking $ \mathbb{R}$ to be one index set and, so, as a starting point to forming some vector space. Taking into account the field properties, is it possible to define a kind of "dimensional derivative"? Assuming this possibility, can we write differential equations that relate the efficiency of some algorithms to the dimension of space? If this approach is possible, the study of algorithmic efficiency according to the vector space size becomes a problem of calculating maximal and minimums values of some functions.

We suppose that answers to these questions already exist, but we hope that from the results obtained in \cite{MudBaseAlg}, at least, they can gain a new approach.
\section{Conclusion}\label{sec:Conc}
We developed a notation that serves as an insight to create original approaches and as a source of new questions. The proof of this is that we made it searching for a redefinition of the matrix concept, but ended finding completely unexpected possibilities. New questions have arisen. Some of which are of purely theoretical interest but also others with considerable applicability in the real world.

In the previous section, we mentioned only a small part of what we already perceive. Realistically, most of these possibilities are unlikely to bring anything new. Because there are so many possibilities, we believe in the chance of something good to emerges from our job. In fact, among so many possibilities, we expect some of them to be productive.

We are working hard with some of them and, hopefully, we will soon get good results. We are already counting on the support of some collaborators and making progress. Curiously, we are obtaining our best results in the context of sparse matrices, even though being a theme so apart of our initial objectives. We will soon display our results.

The heuristic potential of examining old problems with a new look is immeasurable. We hope that this work turns out to be an example of this.
\section*{Acknowledgments}
The author is in debt with many people and organizations, and a fairer list will appear only at the Portuguese version of this work \cite{VetAlg}. Despite this, I cannot fail to thank Ana Maria Machado Matias. Without your enormous efforts, I would never have achieved a university education.
 ----------------------------------------------------------
% Referências bibliográficas
% ----------------------------------------------------------
\bibliographystyle{amsplain}
\bibliography{MudancaDeBase}
% ----------------------------------------------------------

\end{document}